\documentclass{amsart}

\usepackage[utf8]{inputenc}
\usepackage[T1]{fontenc}
\usepackage[english]{babel}
\usepackage{csquotes}
\usepackage{lmodern}
\usepackage{amssymb}
\usepackage[all,cmtip]{xy}
\usepackage{array}
\usepackage{hyperref}
\usepackage{graphicx}
\usepackage{microtype}
\usepackage[numbers,square,sort&compress]{natbib}

\numberwithin{equation}{section}

\newtheorem{theorem}{Theorem}[section]

\newtheorem{proposition}[theorem]{Proposition}
\theoremstyle{definition}
\newtheorem{remark}[theorem]{Remark}
\newtheorem{definition}[theorem]{Definition}
\newtheorem{example}[theorem]{Example}

\newcommand{\ZZ}{\mathbb{Z}}
\newcommand{\QQ}{\mathbb{Q}}
\newcommand{\CC}{\mathbb{C}}
\newcommand{\FF}{\mathbb{F}}
\newcommand{\proj}{\mathbf{P}}
\newcommand{\KK}{\mathcal{K}}
\newcommand{\YY}{\mathcal{Y}}
\newcommand{\JJ}{\mathcal{J}}
\newcommand{\Acal}{\mathcal{A}}
\newcommand{\Zcal}{\mathcal{Z}}
\newcommand{\DD}{\mathcal{D}}
\newcommand{\eps}{\varepsilon}

\DeclareMathOperator{\Jac}{Jac}
\DeclareMathOperator{\Spec}{Spec}
\DeclareMathOperator{\Res}{Res}
\DeclareMathOperator{\Sym}{Sym}
\DeclareMathOperator{\ev}{ev}
\DeclareMathOperator{\divis}{div}

\begin{document}

\bibliographystyle{abbrvnat}

\title[Explicit 2-cover descent for genus 2 curves]{Explicit two-cover descent for genus 2 curves}

\author[D.~R.~Hast]{Daniel Rayor Hast}
\address{Dept.\ of Mathematics and Statistics, Boston University\\Boston, MA 02215, USA}
\email{drhast@bu.edu}


\begin{abstract}
  Given a genus 2 curve $C$ with a rational Weierstrass point defined over a number field, we construct a family of genus 5 curves that realize descent by maximal unramified abelian two-covers of $C$, and describe explicit models of the isogeny classes of their Jacobians as restrictions of scalars of elliptic curves. All the constructions of this paper are accompanied by explicit formulas and implemented in Magma and/or SageMath. We apply these algorithms in combination with elliptic Chabauty to a dataset of 7692 genus 2 quintic curves over $\QQ$ of Mordell--Weil rank 2 or 3 whose sets of rational points have not previously been provably computed. We analyze how often this method succeeds in computing the set of rational points and what obstacles lead it to fail in some cases.
\end{abstract}

\maketitle

\section{Introduction}
\label{section:intro}
Let $C$ be a nice (smooth, projective, geometrically integral) curve over a number field $k$. A central problem in arithmetic geometry is to determine the set of rational points $C(k)$. When $C$ is of genus at least two, by Faltings' theorem, $C(k)$ is a finite set \cite{Faltings83, Faltings84}; however, no general algorithm for provably computing $C(k)$ is currently known. (See \cite{Poonen02} for an overview.)

One common strategy for computing $C(k)$ is \emph{descent}, which involves finding a family of curves $D_\delta$ (with $\delta$ ranging over some computable finite set) together with maps $\varphi_\delta\colon D_\delta \to C$ with the property that $C(k) \subseteq \bigcup_\delta \varphi_\delta(D_\delta(k))$. In many cases, one can construct such families so that the covering curves $D_\delta$ are amenable to other techniques for determining the set of rational points that might not apply directly to $C$.

In this paper, we make explicit a particular descent construction for curves of genus two with a rational Weierstrass point. All the constructions involved are implemented in Magma \cite{Magma} and/or SageMath \cite{sagemath}; source code is available online at \cite{Hast22code}. We then apply these algorithms to all 7692 genus 2 curves over $\QQ$ with a rational Weierstrass point and Mordell--Weil rank 2 or 3 in the database of genus $2$ genus computed by Booker, Sijsling, Sutherland, Voight, and Yasaki \cite{genus2-database} (available at \cite{LMFDB}) and analyze the results.

The general strategy is inspired by prior work of Bruin, Flynn, Stoll, and Wetherell computing rational points on curves using explicit descent constructions; see, for example, \cite{FW99, FW01, Bruin02, BF05}. The constructions of this paper are closely related to those of Bruin and Stoll \cite{Bruin-Stoll} for two-cover descent on arbitrary hyperelliptic curves $C\colon y^2 = f(x)$; we construct genus $5$ quotients of their two-covering curves along with explicit formulas (as a restriction of scalars of an elliptic curve) for a model of the isogeny class of the Jacobian of these quotients. The elliptic curves constructed in this way are isomorphic to those arising from degree four factors of $f$ as discussed in \cite[§8]{Bruin-Stoll}.

From now on, suppose $C$ is of genus $2$ and has a $k$-rational Weierstrass point. Then $C$ has an affine model given by an equation $y^2 = f(x)$, where $f$ has degree exactly $6$, is squarefree, and has a rational root $\alpha$. Thus, elements of $C(k)$ correspond to solutions in $k$ to the equation $y^2 = f(x)$, with the possible addition of two rational points at infinity that are excluded from the affine model. (More precisely, $C$ has two points at infinity, and these points are rational if and only if the leading coefficient of $f$ is a square in $k$.)

Let $J = \Jac(C)$ be the Jacobian variety of $C$. This is an abelian surface whose points correspond to degree zero divisors on $C$ modulo linear equivalence. Embed $C$ in $J$ by the Abel--Jacobi map $P \mapsto (P) - ( (\alpha, 0) )$ associated to the given Weierstrass point $(\alpha, 0)$. Since the chosen base point is a Weierstrass point, multiplication by $-1$ on $J$ induces the hyperelliptic involution $i\colon C \to C$.

A natural $16$-covering of $C$ is given by pullback along $[2]\colon J \to J$, where $[2]$ is multiplication by $2$ in $J$: Let $W = [2]^{-1}(C)$. Then $[2]\colon W \to C$ is a degree $16$ \'etale covering, so by the Riemann--Hurwitz formula \cite[Ch.~IV, Cor.~2.4]{Hartshorne77}, the curve $W$ has genus $17$. In order to compute $C(k)$ via descent using this covering, we would need to do the following:
\begin{enumerate}
  \item compute a finite set of twists $\varphi_\delta\colon W_\delta \to C$ such that $C(k) \subseteq \bigcup_{\delta} \varphi_\delta(W_\delta(k))$; and
  \item compute $W_\delta(k)$ for each twist $W_\delta$.
\end{enumerate}

To make the computations more tractable, it is useful to work instead with a suitable quotient of $W$. Since multiplication by $-1$ on $J$ induces the hyperelliptic involution on $C$, we can lift the hyperelliptic involution to an involution on $W$. Let $Z$ be the quotient of $W$ by this involution. The map $W \to Z$ is ramified exactly at the $2$-torsion points of $J$, of which there are $16$, so $Z$ has genus $5$ by the Riemann--Hurwitz formula. (Another model for this curve can be constructed using the methods of \cite[§3.1]{Bruin02}; we choose this approach to emphasize the connection with Kummer surfaces.)

The purpose of this paper is to give explicit, computationally tractable formulas for $Z$ and its Jacobian (and their twists), along with the associated maps realizing the correspondence with $C$; to apply these constructions in combination with the elliptic Chabauty method to the aforementioned large dataset of curves; and to determine what the obstructions are in the cases where it does not succeed. The key ingredient is to embed (twists of) $Z$ in (twists of) the desingularized Kummer surface of $J$. Our primary references for the requisite explicit descriptions of Kummer surfaces and their twists are \cite{Cassels-Flynn} and \cite{FTvL12}.

In section \ref{section:curves}, we provide the necessary background on desingularized twisted Kummer surfaces, construct the canonical embedding of $Z$ and its twists as hyperplane sections of these surfaces, and describe the primes of bad reduction. In section \ref{section:map}, we prove an explicit formula for the twisted duplication map and describe its ramification divisor. In section \ref{section:genus-one}, we construct a map to a genus one curve through which the twisted duplication map factors, which supplies the necessary data to apply the elliptic Chabauty method \cite{Bruin03}; we also use this to give an explicit model for the Jacobian of $Z$ up to isogeny. In section \ref{section:results}, we report on the results of applying this method to the aforementioned dataset of 7692 curves; the method succeeds for 1045 of these curves, and we analyze the obstacles encountered for the remaining curves. Finally, in section \ref{section:examples}, we analyze the method and results in detail for several examples.

\section{Genus 5 curves in twisted Kummer surfaces}
\label{section:curves}
Let $k$ be a field not of characteristic $2$. Let $C\colon y^2 = f(x)$ be a genus $2$ curve over $k$ with $\deg(f) = 6$ such that $C$ has a $k$-rational Weierstrass point $(\alpha, 0)$. (Although such a curve does have a quintic model over $k$, we work with sextic models in order to use the explicit description of desingularized twisted Kummer surfaces outlined below.) Let $i\colon C \to \proj^1$ be the canonical map. Let $f_0, f_1, \dots, f_6 \in k$ and $\gamma_1, \gamma_2, \dots, \gamma_6 \in k$ such that
\[
f(x) = \sum_{i=0}^{6} f_i x^i = (x - \alpha) (\gamma_1 + \gamma_2 x + \dots + \gamma_6 x^5).
\]

Let $J$ be the Jacobian variety of $C$. Let $L = k[X]/\langle f(X) \rangle$, and let $\delta = \sum_{i=0}^{5} d_i X^i \in L^*$ be arbitrary. When $k$ is a global or local field, Flynn, Testa, and van Luijk \cite[§7]{FTvL12} construct a twist $\pi_\delta\colon A_\delta \to J$ of the multiplication-by-$2$ map $[2]\colon J \to J$, depending up to isomorphism only on the class of $\delta$ in $L^*/{L^*}^2 k^*$, whose class in $H^1(k, J[2])$ maps to $\delta$ under the Cassels map $\mu\colon J(k)/2J(k) \to L^*/{L^*}^2 k^*$. They also show that every two-covering of $J$ that has a $k$-rational point arises in this way \cite[Prop.~2.15]{FTvL12}, and so if $\Delta \subseteq L^*$ is any subset whose image in $L^*/{L^*}^2 k^*$ contains the image of the Cassels map $\mu$, we have
\[
J(k) = \bigcup_{\delta \in \Delta} \pi_\delta(A_\delta(k)).
\]

Each surface $A_\delta$ is equipped with a natural involution $\iota_\delta\colon A_\delta \to A_\delta$ lifting $[-1]\colon J \to J$. The (twisted) Kummer surfaces $\KK_\delta = A_\delta/\langle \iota_\delta \rangle$ have $16$ simple nodes. For computational purposes, it turns out to be more convenient to work with their minimal desingularizations $\YY_\delta$. Let $p_\delta\colon A_\delta \dashrightarrow \YY_\delta$ be the rational quotient map. Let $W_\delta = \pi_\delta^{-1}(C)$, where we embed $C$ in $J$ by the Abel--Jacobi map $P \mapsto (P) - ( (\alpha, 0) )$. By the Riemann--Hurwitz formula, $W_\delta$ has genus $17$. Let
\[
Z_\delta = p_\delta(W_\delta),
\]
and let
\[
\bar{\pi}_\delta\colon Z_\delta \to \proj^1
\]
be the map defined by sending a general point $Q \in Z_\delta$ to $i(\pi_\delta(\tilde{Q}))$, where $\tilde{Q} \in A_\delta$ is such that $p_\delta(\tilde{Q}) = Q$. (Since $C$ is embedded in $J$ via an Abel--Jacobi map whose base point is a Weierstrass point, the hyperelliptic involution on $C$ lifts to $\iota_\delta$ on $A_\delta$, so this is well-defined.)

Thus $Z_\delta$ fits into a commutative diagram
\begin{equation}
  \label{eqn:main-diagram}
  \xymatrix{
    & W_\delta \ar[dl]_{p_\delta} \ar[dr]^{\pi_\delta} \\
    Z_\delta \ar[dr]^{\bar{\pi}_\delta} & & C \ar[dl]_{i} \\
    & \proj^1
  }
  \qquad\qquad
  \xymatrix{
    & A_\delta \ar@{-->}[dl]_{p_\delta} \ar[dr]^{\pi_\delta} \\
    \YY_\delta \ar[dr]^{\bar{\pi}_\delta} & & J \ar[dl] \\
    & \KK 
  }
\end{equation}
where each curve in the left diagram embeds into the corresponding surface in the right diagram, and $\KK = J/[-1]$ is the Kummer surface of $J$.

We reproduce here the model of $\YY_\delta$ constructed in \cite[§4]{FTvL12} as the complete intersection of three quadrics in $\proj^5 = \proj(L)$ (recall that $L = k[X]/\langle f(X) \rangle$), which we have implemented in both SageMath and Magma.

\begin{definition}
\label{def:kummer}
Write $f(X) = \sum_{i=0}^{6} f_i X^i$ and $\delta = \sum_{i=0}^{5} d_i X^i \in L$. Let
\[
R = \begin{pmatrix}
  0 & 0 & 0 & 0 & 0 & -f_0 f_6^{-1} \\
  1 & 0 & 0 & 0 & 0 & -f_1 f_6^{-1} \\
  0 & 1 & 0 & 0 & 0 & -f_2 f_6^{-1} \\
  0 & 0 & 1 & 0 & 0 & -f_3 f_6^{-1} \\
  0 & 0 & 0 & 1 & 0 & -f_4 f_6^{-1} \\
  0 & 0 & 0 & 0 & 1 & -f_5 f_6^{-1}
\end{pmatrix} \qquad \text{and} \qquad T = \begin{pmatrix}
  f_1 & f_2 & f_3 & f_4 & f_5 & f_6 \\
  f_2 & f_3 & f_4 & f_5 & f_6 & 0 \\
  f_3 & f_4 & f_5 & f_6 & 0 & 0 \\
  f_4 & f_5 & f_6 & 0 & 0 & 0 \\
  f_5 & f_6 & 0 & 0 & 0 & 0 \\
  f_6 & 0 & 0 & 0 & 0 & 0
\end{pmatrix}.
\]
Let $g_1, \dots, g_6$ be the basis of $L$ defined by
\begin{align*}
  g_1 &= f_1 + f_2 X + f_3 X^2 + f_4 X^3 + f_5 X^4 + f_6 X^5, \\
  g_2 &= f_2 + f_3 X + f_4 X^2 + f_5 X^3 + f_6 X^4, \\
  g_3 &= f_3 + f_4 X + f_5 X^2 + f_6 X^3, \\
  g_4 &= f_4 + f_5 X + f_6 X^2, \\
  g_5 &= f_5 + f_6 X, \\
  g_6 &= f_6,
\end{align*}
and let $v_1, \dots, v_6$ be the dual basis of $\hat{L}$. For $j \geq 0$, let $Q_j^{(\delta)}$ be the quadratic form corresponding to the symmetric matrix $\sum_{i=0}^{5} f_6 d_i R^{i+j} T$ in the basis $v_1, \dots, v_6$. Let $\YY_\delta \subset \proj(L)$ be defined by
\[
Q_0^{(\delta)} = Q_1^{(\delta)} = Q_2^{(\delta)} = 0.
\]
\end{definition}
Flynn, Testa, and van Luijk show that $\YY_\delta$ is indeed the minimal desingularization of $\KK_\delta$ \cite[§7]{FTvL12}.

\begin{theorem}
  \label{thm:embedding}
  Let $\ev_\alpha\colon L \to k$ be the homomorphism defined by $\ev_\alpha(\xi) = \xi(\alpha)$. The curve $Z_\delta$ is the intersection of $\YY_\delta$ with the hyperplane $\proj(\ker(\ev_\alpha)) \subset \proj(L)$, which is given in coordinates by
  \[
  \gamma_1 v_1 + \gamma_2 v_2 + \gamma_3 v_3 + \gamma_4 v_4 + \gamma_5 v_5 + \gamma_6 v_6 = 0,
  \]
  where as above, $\gamma_1, \dots, \gamma_6$ are the coefficients of $f(x)/(x - \alpha)$.
\end{theorem}
\begin{proof}
  Note that $\ev_\alpha$ is well-defined since $f(\alpha) = 0$. The first step is to understand how the base $\proj^1$ embeds into $\KK$. Recall that a \emph{trope} of a quartic surface $\KK \subset \proj^3$ is a tangent plane that intersects the surface at a conic (with multiplicity two). A Kummer surface has exactly $16$ tropes, which are the projective duals of the $16$ nodes \cite[§3.7]{Cassels-Flynn}\cite[§8-9]{Hudson}.

  The description of six of the tropes of the Kummer surface given in \cite[§7.6]{Cassels-Flynn} shows that the image of $\proj^1$ in $\KK$ corresponding to the Abel--Jacobi map with base point $(\alpha, 0)$ is contained in the trope $T_\alpha$ with equation $\alpha^2 \kappa_1 - \alpha \kappa_2 + \kappa_3 = 0$, where $\kappa_1, \kappa_2, \kappa_3, \kappa_4$ are the coordinates of the usual embedding of $\KK$ as a quartic surface in $\proj^3$ (as described in \cite[§3.1]{Cassels-Flynn}, using the letters $\xi_i$ instead of $\kappa_i$). (Strictly speaking, $T_\alpha$ intersects $\KK$ in \emph{twice} the conic corresponding to the base $\proj^1$; we also denote this conic by $T_\alpha$.)

  Instead of constructing $Z_\delta$ as the quotient of a genus $17$ curve embedded in $J$, we compute $Z_\delta$ using maps between the twisted Kummer surfaces. The preimage of $T_\alpha$ under the map $\bar{\pi}_\delta\colon \YY_\delta \to \KK$ also contains some exceptional divisors, which we wish to omit when constructing $Z_\delta$, so we instead consider the condition for a \emph{general} element of $\YY_\delta$ to map to $T_\alpha$ via the map $\bar{\pi}_\delta$.

  Let $P \in \YY_\delta$ be an arbitrary point not contained in the locus of indeterminacy of the rational map $p_\delta\colon A_\delta \dashrightarrow \YY_\delta$. Let $\xi \in L$ be an arbitrary lift of $P$ from $\proj(L)$ to $L$. We will show that $\bar{\pi}_\delta(P) \in T_\alpha$ if and only if $\xi(\alpha) = 0$.

  We first treat the untwisted case $\delta = 1$. Let $D = ( (x_1, y_1) ) + ( (x_2, y_2) ) - K_C$ such that $p_1([D]) = [\pm D] = P$. As explained in the paragraph preceding \cite[Prop.~4.11]{FTvL12}, the $x$-coordinates of the points $R_1$ and $R_2$ such that $2D \sim R_1 + R_2 - K_C$ are the roots of the quadratic polynomial $H(X)$ corresponding to $\xi^2$. (This does not depend on the choice of lift $\xi \in L$ since choosing a different lift multiplies $H$ by an element of $k^*$, which does not change the roots.) The condition that $\pm 2D$ is contained in $T_\alpha$ is exactly that one of the $x$-coordinates of $R_1$ and $R_2$ is $\alpha$, i.e., that $\alpha$ is a root of $H$, or equivalently, that $\xi^2(\alpha) = 0$, which is the case if and only if $\xi(\alpha) = 0$.

  Now we handle the twisted case. Let $D = ( (x_1, y_1) ) + ( (x_2, y_2) ) - K_C$ such that $p_\delta([D]) = P$. Let $k^s$ be a separable closure of $k$, let $L^s = L \otimes_k k^s$, and let $\eps \in L^s$ such that $\eps^2 = \delta$. Then $\delta \xi^2 = (\eps \xi)^2$, so $\eps \xi \in \YY_1$. Let $D' \in J$ such that the image of $D'$ in $\YY_1$ is $\eps \xi$. By \cite[§7]{FTvL12}, we have $\bar{\pi}_\delta = [2] \circ g$, where $g$ is defined by multiplication by $\eps$ in $L$. So
  \[
  \bar{\pi}_\delta(\xi) = \bar{\pi}_\delta(g^{-1}(\eps \xi)) = [2](\eps \xi),
  \]
  and lifting to the Jacobian, the divisor class corresponding to $\bar{\pi}_\delta(\xi)$ is equal to $[2D']$, i.e., $\pi_\delta([D]) = [2D']$. As in the previous paragraph, the roots $r_1, r_2$ of the quadratic polynomial $H$ such that $\delta \xi^2 \equiv H \pmod{f}$ are the $x$-coordinates of points $R_1, R_2 \in C$ such that $(R_1) + (R_2) - K_C \sim 2D'$. Thus, $\bar{\pi}_\delta(P) \in T_\alpha$ if and only if $\alpha$ is a root of $H$, which is equivalent to $\delta(\alpha) \xi^2(\alpha) = 0$. Since $\delta \in L^*$, we have $\delta(\alpha) \neq 0$, so this is equivalent to $\xi(\alpha) = 0$.

  Represent $\xi$ in the basis $g_1, \dots, g_6$ as $\xi = \sum_{i=1}^{6} v_i(\xi) g_i$. Then
  \[
  \xi(\alpha) = \sum_{i=1}^{6} g_i(\alpha) v_i(\xi),
  \]
  so in the basis $v_1, \dots, v_6$, the condition $\xi(\alpha) = 0$ becomes
  \[
  g_1(\alpha) v_1 + g_2(\alpha) v_2 + g_3(\alpha) v_3 + g_4(\alpha) v_4 + g_5(\alpha) v_5 + g_6(\alpha) v_6 = 0.
  \]
  It follows immediately from the definitions of $g_1, \dots, g_6$ that
  \[
  f(x) = (x - \alpha) (g_1(\alpha) + g_2(\alpha) x + g_3(\alpha) x^2 + g_4(\alpha) x^3 + g_5(\alpha) x^4 + g_6(\alpha) x^5),
  \]
  i.e., $\gamma_i = g_i(\alpha)$ for each $i \in \{1, \dots, 6\}$, so $Z_\delta$ is in fact the hyperplane section of $\YY_\delta$ whose coefficients in the basis $v_1, \dots, v_6$ are the coefficients of the polynomial $f(x)/(x - \alpha)$.
\end{proof}

\begin{proposition}
  \label{prop:smooth}
  The curve $Z_\delta$ is smooth, has genus $5$, and is canonically embedded in $\proj(\ker(\ev_{\alpha})) \cong \proj^4$.
\end{proposition}
\begin{proof}
  Let $q_\delta\colon A_\delta \to \KK_\delta$ be the quotient map. The ramification divisor of $q_\delta$ is $\pi_\delta^{-1}(0)$. At each point $P \in \pi_\delta^{-1}(0)$, since $q_\delta$ has degree $2$, the point $q_\delta(P)$ is either nonsingular or a simple node. The map $\YY_\delta \to \KK_\delta$ is given by blowing up at $q_\delta(\pi_\delta^{-1}(0))$, which desingularizes any simple nodes, so $Z_\delta$ (being the proper transform of $q_\delta(W_\delta)$) is smooth.

  By Theorem \ref{thm:embedding}, the curve $Z_\delta$ is a complete intersection of three quadrics in $\proj^4$, so $Z_\delta$ is a canonical curve of genus $5$ (cf.\ \cite[Ch.~IV, Ex.~5.5.3]{Hartshorne77}).
\end{proof}

\begin{proposition}
  \label{prop:good-reduction}
  Suppose that $k$ is a local field with residue field $\FF_q$, and that $C/k$ has good reduction. If $q$ is odd, then $Z_\delta$ also has good reduction.
\end{proposition}
\begin{proof}
  Write $q = p^n$, where $p$ is an odd prime. By \cite[Expos\'e X, Cor.~3.9]{SGA1}, specialization to $\FF_q$ induces an isomorphism between the prime-to-$p$ parts of the \'etale fundamental groups of $C$ and the special fiber $\bar{C}/\FF_q$. Thus $W_\delta$, being a degree $16$ \'etale cover of $C$, also has good reduction. Euler characteristic (and hence also arithmetic genus) are locally constant in proper flat families \cite[§5, Cor.~1]{Mumford}, so Proposition \ref{prop:smooth} implies that the special fiber $\bar{Z}_\delta/\FF_q$ has arithmetic genus $5$, hence geometric genus at most $5$. Since $p$ is odd, the quotient map $\bar{W}_\delta \to \bar{Z}_\delta$ is tamely ramified, so the Riemann--Hurwitz formula implies that $\bar{Z}_\delta$ has geometric genus exactly $5$ and thus is smooth over $\FF_q$.
\end{proof}

\section{The twisted duplication map}
\label{section:map}
In this section, we give explicit formulas for the map $\bar{\pi}_\delta\colon Z_\delta \to \proj^1$ induced by the twisted duplication map. We also give an explicit description of the ramification divisor of this map.

\begin{theorem}
  \label{thm:map}
  For all $P \in Z_\delta$, we have
  \[
  \bar{\pi}_\delta(P) = \left( -(f_5 + f_6 \alpha) Q_3^{(\delta)}(P) - f_6 Q_4^{(\delta)}(P) : f_6 Q_3^{(\delta)}(P) \right) \in \proj^1.
  \]
\end{theorem}
\begin{proof}
As in \cite[§4]{FTvL12}, let $C_0^{(\delta)}, \dots, C_5^{(\delta)} \in \Sym^2(\hat{L})$ be quadratic forms such that $C_j^{(\delta)}(z) = p_j(\delta z^2)$ for $z \in L$, where $p_j$ gives the coefficient of $X^j$. We have
\[
f_6 \cdot \begin{pmatrix}
  C_0^{(\delta)} & C_1^{(\delta)} & \dots & C_5^{(\delta)}
\end{pmatrix} = \begin{pmatrix}
  Q_0^{(\delta)} & Q_1^{(\delta)} & \dots & Q_5^{(\delta)}
\end{pmatrix} \cdot T,
\]
where $T$ is the matrix defined in Definition \ref{def:kummer}, so that in particular
\begin{align*}
  f_6 C_1^{(\delta)} &= f_2 Q_0^{(\delta)} + f_3 Q_1^{(\delta)} + f_4 Q_2^{(\delta)} + f_5 Q_3^{(\delta)} + f_6 Q_4^{(\delta)}, \\
  f_6 C_2^{(\delta)} &= f_3 Q_0^{(\delta)} + f_4 Q_1^{(\delta)} + f_5 Q_2^{(\delta)} + f_6 Q_3^{(\delta)}.
\end{align*}
Thus, taking into account that $Q_j^{(\delta)}$ vanishes on $\YY_\delta$ for $j \in \{0, 1, 2\}$, we have
\[
( -(f_5 + f_6 \alpha) Q_3^{(\delta)}(P) - f_6 Q_4^{(\delta)}(P) : f_6 Q_3^{(\delta)}(P) ) = ( -C_1^{(\delta)}(P) - \alpha C_2^{(\delta)}(P) : C_2^{(\delta)}(P) ).
\]
Moreover, $C_3^{(\delta)} = C_4^{(\delta)} = C_5^{(\delta)} = 0$ on $\YY_\delta$.

Let $\xi \in L$ be a lift of $P \in Z_\delta \subset \proj(L)$. By construction of $\YY_\delta$, we have
\[
\delta \xi^2 \equiv C_2^{(\delta)}(\xi) X^2 + C_1^{(\delta)}(\xi) X + C_0^{(\delta)}(\xi) \pmod{f}.
\]
As explained in the proof of Theorem \ref{thm:embedding}, the roots of this quadratic polynomial are the $x$-coordinates of points of the divisor in $J$ corresponding to $\bar{\pi}_\delta(P)$. Moreover, since $P \in Z_\delta$, one of these roots is $\alpha$. Thus, in the affine patch where the second coordinate of $\proj^1$ is nonzero, writing $\bar{\pi}_\delta(P) = (r : 1)$, we have
\[
C_2^{(\delta)}(\xi) X^2 + C_1^{(\delta)}(\xi) X + C_0^{(\delta)}(\xi) = c (X - \alpha) (X - r)
\]
for some nonzero $c \in k^s$. Comparing coefficients, we obtain $C_2^{(\delta)}(\xi) = c$ and $C_1^{(\delta)}(\xi) = -c(\alpha + r)$, so
\[
r = \frac{-C_1^{(\delta)}(\xi) - \alpha C_2^{(\delta)}(\xi)}{C_2^{(\delta)}(\xi)}.
\]
This gives the desired formula for $\bar{\pi}_\delta(P)$. Finally, we have $\bar{\pi}_\delta(P) = (1 : 0)$ if and only if $C_2^{(\delta)}(\xi) = 0$, completing the proof.
\end{proof}

\begin{theorem}
  \label{thm:ramification}
  Let $\Omega \subset k^s$ be the set of roots of $f$. The branch locus of $\bar{\pi}_\delta\colon Z_\delta \to \proj^1$ is $\Omega \setminus \{\alpha\}$. For each $\omega \in \Omega \setminus \{\alpha\}$, we have
  \[
  \bar{\pi}_\delta^{-1}(\omega) = Z_\delta \cap \proj(\ker(\ev_\omega)) \subset \proj(L),
  \]
  which consists of $8$ geometric points, each of ramification index $2$.
\end{theorem}
\begin{proof}
  Observe that $\pi_\delta\colon W_\delta \to C$ is \'etale, the branch locus of $i\colon C \to \proj^1$ is $\Omega$, and the branch locus of $p_\delta\colon W_\delta \to Z_\delta$ is $\bar{\pi}_\delta^{-1}(\alpha)$, with all ramification indices in the preimage of the branch locus equal to $2$. Thus, commutativity of diagram \eqref{eqn:main-diagram} implies that the branch locus of $\bar{\pi}_\delta$ is $\Omega \setminus \{\alpha\}$, and for each $\omega \in \Omega \setminus \{\alpha\}$, the preimage $\bar{\pi}_\delta^{-1}(\omega)$ consists of $8$ geometric points of ramification index $2$.

  The remaining claim that $\bar{\pi}_\delta^{-1}(\omega)$ is the hyperplane section of $Z_\delta$ given by intersection with $\proj(\ker(\ev_\omega))$ follows from the description of $\bar{\pi}_\delta$ given in the proofs of Theorems \ref{thm:embedding} and \ref{thm:map}: For $\xi \in L$ lifting a point $P \in Z_\delta$, we have $\bar{\pi}_\delta(P) = (\omega : 1)$ if and only if the quadratic polynomial defining $\delta \xi^2$ has roots $\alpha$ and $\omega$, which is equivalent to the condition $\xi(\alpha) = \xi(\omega) = 0$, i.e., $P$ is in the kernel of both the evaluation maps $\ev_\alpha$ (which defines $Z_\delta$ as a hyperplane section of $\YY_\delta$) and $\ev_\omega$, as was to be shown.
\end{proof}

\section{Maps to genus one curves}
\label{section:genus-one}
We now construct a map to a genus one curve through which the twisted duplication map factors, and prove that this map induces an isogeny from the Jacobian of $Z_\delta$ to the restriction of scalars of the Jacobian of this genus one curve. These genus one curves are geometrically Prym varieties \cite[Ch.~12]{BL04} associated to double coverings of $C$. This is a substantial motivation for the constructions of this paper, since a restriction of scalars of an elliptic curve is much more computationally accessible than a general Jacobian variety of the same dimension.

\begin{theorem}
\label{thm:genus-one}
Let $K = k(\omega)$, where $\omega \in k^s$ is a root of $f$ and $\omega \neq \alpha$. Write $f(x) = (x - \alpha) (x - \omega) h(x)$, let $H(x, z)$ be the homogenization of $h(x)$ with respect to $z$, and let $\beta_1, \beta_2, \beta_3, \beta_4 \in k^s$ be the roots of $h$. Let $Y_{\alpha, \omega} = \ev_{\beta_1} \cdot \ev_{\beta_2} \cdot \ev_{\beta_3} \cdot \ev_{\beta_4}$, where $\ev_{\beta_j} = \sum_{i=1}^{6} g_i(\beta_j) v_i$ is given by evaluation at $\beta_j$. (Note that $Y_{\alpha, \omega}$ is a quartic form over $K$.)

Define a curve $D_{\delta, \omega} \subset \proj(1, 2, 1)$ in weighted projective space by the equation
\[
Y_{\alpha, \omega}(\delta) y^2 = h(\alpha) H(x, z).
\]
Define a map $\varphi\colon Z_\delta \to \proj(1, 2, 1)$ over $K$ by
\[
\varphi(P) = \left( -(f_5 + f_6 \alpha) Q_3^{(\delta)}(P) - f_6 Q_4^{(\delta)}(P) : f_6^3 Y_{\alpha, \omega}(P) : f_6 Q_3^{(\delta)}(P) \right).
\]
Then the image of $\varphi$ is $D_{\delta, \omega}$, and the following diagram commutes:
\[
\xymatrix{
  Z_\delta \ar[r]^{\varphi} \ar[dr]_{\bar{\pi}_\delta} & D_{\delta, \omega} \ar[d]^{x} \\
  & \proj^1
}
\]
\end{theorem}
\begin{proof}
  Since $\ev_{\beta_j}$ is a ring homomorphism for each $j$, the quartic form $Y_{\alpha, \omega}$ is multiplicative with respect to $L$, i.e., $Y_{\alpha, \omega}(\xi \eta) = Y_{\alpha, \omega}(\xi) Y_{\alpha, \omega}(\eta)$ for all $\xi, \eta \in L$. As proved in Theorem \ref{thm:map}, for all $\xi \in L^s$ lifting a point $P \in Z_\delta(k^s)$, we have 
  \[
  f_6 \delta \xi^2 = (X - \alpha) \left( f_6 Q_3^{(\delta)}(\xi) X + (f_5 + f_6 \alpha) Q_3^{(\delta)}(\xi) + f_6 Q_4^{(\delta)}(\xi) \right).
  \]
  Putting these together, we obtain
  \begin{align*}
    &Y_{\alpha, \omega}(\delta) (f_6^3 Y_{\alpha, \omega}(\xi))^2 = f_6^6 Y_{\alpha, \omega}(\delta \xi^2) = f_6^2 Y_{\alpha, \omega}(f_6 \delta \xi^2) \\
    &= f_6 \prod_{j=1}^{4} (\beta_j - \alpha) \cdot f_6 \prod_{j=1}^{4} \left( f_6 Q_3^{(\delta)}(\xi) \beta_j + (f_5 + f_6 \alpha) Q_3^{(\delta)}(\xi) + f_6 Q_4^{(\delta)}(\xi) \right) \\
    &= h(\alpha) H(-(f_5 + f_6 \alpha) Q_3^{(\delta)}(\xi) - f_6 Q_4^{(\delta)}(\xi), f_6 Q_3^{(\delta)}(\xi)).
  \end{align*}
  Thus $\varphi(Z_\delta) \subseteq D_{\delta, \omega}$. Since $\varphi$ is non-constant and $D_{\delta, \omega}$ is an irreducible curve, $\varphi(Z_\delta) = D_{\delta, \omega}$. Commutativity of the diagram is immediate from the formulas.
\end{proof}

\begin{remark}
  Theorem \ref{thm:ramification} gives another perspective on Theorem \ref{thm:genus-one} in terms of divisors: Denote $\varphi = (\varphi_x : \varphi_y : \varphi_z)$. By Theorem \ref{thm:ramification}, for each root $\beta$ of $h$,
  \[
  \bar{\pi}_\delta^*( (\beta) - (\infty) ) = \divis(\ev_\beta^2/\varphi_z).
  \]
  Consider the rational functions $R = \varphi_x/\varphi_z$ and $S = \varphi_y/\varphi_z^2$. Then
  \[
  \divis(h \circ R) = \bar{\pi}_\delta^*(\divis(h)) = \bar{\pi}_\delta^*( (\beta_1) + (\beta_2) + (\beta_3) + (\beta_4) - 4 (\infty) ) = \divis(S^2).
  \]
  So $S^2$ is a scalar multiple of $h \circ R$; comparing their values at any point outside the divisor of zeroes and poles yields Theorem \ref{thm:genus-one}. (This is how the author initially discovered the formulas.)
\end{remark}

\begin{remark}
  \label{remark:elliptic-curves}
  If $D_{\delta, \omega}(K)$ is empty, then so is $Z_\delta(K)$. If $D_{\delta, \omega}(K)$ is nonempty, then $D_{\delta, \omega}$ is isomorphic to an elliptic curve $E_\delta = \Jac(D_{\delta, \omega})$ over $K$. In the latter case, if $k = \QQ$, then Theorem \ref{thm:genus-one} provides exactly the requisite data to compute $Z_\delta(\QQ)$ using the elliptic Chabauty method, provided that we can compute generators for the Mordell--Weil group $E_\delta(K)$ and that the rank of $E_\delta(K)$ is less than $[K : \QQ]$.

  One can find an upper bound on the rank of $E_\delta(K)$ by computing the $2$-Selmer group (and this is the method we use in the examples of the next section). This requires computing the class group of $K[x]/\langle \eta_\delta(x) \rangle$, where we write $E_\delta\colon y^2 = \eta_\delta(x)$. This is often computationally expensive unless we assume Bach's bound \cite{Bach90} on the norm of prime ideals needed to generate the class group, which is conditional on the generalized Riemann hypothesis (GRH). However, since varying $\delta$ only changes $D_{\delta, \omega}$ by a quadratic twist, the elliptic curves $E_\delta$ also only differ by a quadratic twist, so the quotient algebra $K[x]/\langle \eta_\delta(x) \rangle$ does not depend on $\delta$. Thus, the expensive class group computation need only be carried out once for the whole twist family, rather than for each twist individually.
\end{remark}

We now relate the above genus one curves to the Jacobian of $Z_\delta$.

\begin{theorem}
  \label{thm:jacobian}
  Let $g(x) = f(x)/(x - \alpha)$, let $B = k[w]/\langle g(w) \rangle$, let $K_1, \dots, K_r$ be fields over $k$ such that $B \cong K_1 \times \dots \times K_r$, and let $\omega_i$ be the image of $w$ in $K_i$ for each $i$. Let $D_\delta = \coprod_{i=1}^{r} D_{\delta, \omega_i}$ be the curve from Theorem \ref{thm:genus-one} considered as a curve over $B$, let $\varphi\colon Z_\delta \to D_\delta$ be the corresponding morphism over $B$, and let $E_\delta = \Jac(D_\delta) = \coprod_{i=1}^{r} \Jac(D_{\delta, \omega_i})$. Then the induced $k$-morphism of abelian varieties
  \[
  \Jac(Z_\delta) \to \Res_k^B(E_\delta) \cong \prod_{i=1}^{r} \Res_k^{K_i}(\Jac(D_{\delta, \omega_i}))
  \]
  is an isogeny.
\end{theorem}
\begin{proof}
  Our strategy is to consider universal families of curves and abelian varieties corresponding to the above situation, observe that the properties of interest are deformation-invariant, and deform the problem to a more computationally tractable case.

  Let $S = \Spec A$ be the space parametrizing triples $(g, \alpha, \delta) \in k[w] \times k \times k[X]$ such that $g$ is a monic squarefree quintic polynomial with $g(\alpha) \neq 0$, the degree of $\delta$ is at most $5$, and $\delta$ is invertible modulo $(X - \alpha) \cdot g(X)$. Let $P \in A[w]$ be the generic monic quintic polynomial, and let $T = \Spec A[w]/\langle P(w) \rangle$. Let $\Zcal \to S$ and $\DD \to T$ be the relative curves whose fibers above a point $(g, \alpha, \delta) \in S$ are the genus $5$ curve $Z_\delta$ and the genus $1$ curve $D_\delta$, respectively, that are associated to the twisting parameter $\delta$ for the hyperelliptic curve $y^2 = (x - \alpha) g(x)$. Let $\JJ \to S$ be the relative Jacobian variety of $\Zcal$, and let $\Acal = \Res_S^T(\Jac(\DD))$, which exists as a scheme since $T \to S$ is \'etale.

  The formulas of Theorem \ref{thm:genus-one} define a $T$-morphism $\Zcal \times_S T \to \DD$, which induces a homomorphism of abelian $S$-schemes $\Phi\colon \JJ \to \Acal$. By \cite[Lemma 6.12]{GIT}, the homomorphism $\JJ \to \Phi(\JJ)$ is flat. The kernel $\ker(\Phi)$ is the fiber product of $\Phi$ with the unit section $S \to \Acal$, so $\ker(\Phi)$ is a flat proper $S$-group scheme since flatness and properness are preserved by base change. By \cite[Expos\'{e} VI\textsubscript{B}, Cor.~4.3]{SGA3-1}, since $S$ is also connected, the fibers of the map $\ker(\Phi) \to S$ all have the same dimension. Moreover, if $\ker(\Phi) \to S$ has relative dimension zero, then $\ker(\Phi)$ is a finite flat $S$-group scheme by \cite[Thm.~8.11.1]{EGA-IV.3}. Thus, we can compute the relative dimension of $\Phi$ on any fiber, and if $\Phi$ is an isogeny, we can also compute its degree on any fiber.

  Let $g \in k[w]$ such that $s := (g, 0, 1) \in S(k)$ and $g$ splits completely over $k$. Let $\omega_1, \dots, \omega_5$ be the roots of $g$. By functoriality of restriction of scalars,
  \[
  \Acal_s \cong \Res_k^{k^5}(\Jac(\DD_s)) \cong \prod_{i=1}^{5} E_i,
  \]
  where $E_i$ is the Jacobian of the genus $1$ curve defined by $y^2 = -\omega_i^{-1} g(0) g(x)/(x - \omega_i)$. Furthermore, choose $g$ so that the elliptic curves $E_i$ are pairwise non-isogenous. (If no such polynomial $g$ is defined over $k$, it is harmless to extend scalars to a larger field, since this preserves both dimension and degree.)

  The composition of the map $\Phi_s\colon \JJ_s \to \prod_{i=1}^{5} E_i$ with any of the five projection maps $\prod_{i=1}^{5} E_i \to E_j$ is induced by the map $\varphi$ of Theorem \ref{thm:genus-one} (with $\omega = \omega_i$), hence is surjective. Thus, the image of $\Phi_s$ contains an elliptic curve isogenous to $E_j$ for each $j$. Since the $E_j$ are pairwise non-isogenous, this implies that $\Phi_s$ is surjective. Since $\dim \JJ_s = 5$, this means $\Phi_s$ is an isogeny.
\end{proof}

\begin{remark}
  An analytic computation using Magma's algorithms for period matrices of Riemann surfaces shows that in characteristic zero, up to numerical error, $\Jac(Z_\delta)$ is isogenous to $\Res_k^B(E_\delta)$ via a degree $32$ isogeny. The above proof shows that it suffices to compute the degree for any one example, and we then apply the algorithms to the example $f(x) = \prod_{\omega=-2}^{3} (x - \omega)$. Given big period matrices $P_1$ and $P_2$ of the corresponding Riemann surfaces, the \texttt{IsIsogenousPeriodMatrices} function in Magma computes matrices $M \in M_5(\CC)$ and $N \in M_{10}(\ZZ)$ such that $MP_1 = P_2 N$. This defines an isogeny of degree $\det(N)$ between the corresponding complex tori; we compute $\det(N) = 32$ for this example.
\end{remark}

\section{Results}
\label{section:results}

Using Magma v2.26-10 and SageMath 9.3 on Boston University's Shared Computing Cluster \cite{SCC}, a heterogeneous Linux-based computing cluster with approximately 21000 cores, the above algorithms were applied to all 7692 genus 2 curves over $\QQ$ in \cite{genus2-database} that have at least one rational Weierstrass point and Mordell--Weil rank at least 2. Each of these curves has Mordell--Weil rank 2 or 3, so Chabauty's method \cite{Chabauty41, MP12} is not directly applicable. Table \ref{tab:outcomes} summarizes the results.

\begin{table}[ht]
  \centering
  \caption{Outcomes of running the code on the dataset of 7692 genus 2 curves.}
  \begin{tabular}{l|r|r}
    \textbf{Outcome} & \textbf{Count} & \textbf{Percent} \\
    \hline
    Success & 1045 & 13.6\% \\
    Apparent failure of Hasse principle & 2120 & 27.6\% \\
    Mordell--Weil rank too high & 802 & 10.4\% \\
    Unable to compute Mordell--Weil group & 2271 & 29.5\% \\
    Exceeded time or memory limits & 1685 & 21.9\% \\
    Miscellaneous error & 19 & 0.2\%
  \end{tabular}
  \label{tab:outcomes}
\end{table}

By \enquote{apparent failure of the Hasse principle}, we mean that one of the genus 5 covering curves $Z_\delta$ is locally solvable, but a point search did not find any rational points on it. Note that the counts add up to more than 7692 because multiple obstructions were found for some curves---for example, a genus 5 curve might map to two different elliptic curves, one of which has too high rank and the other for which Magma cannot compute the Mordell--Weil group.

The raw data is publicly available on GitHub \cite{Hast22data}. The data is in the format of a JSON file for each curve, containing the results of the computation as well as the necessary data to reproduce some of the intermediate steps. (This data includes, for example, coefficients of all curves constructed, as well as coordinates of generators of any Mordell--Weil groups computed.)

The computations of Mordell--Weil groups of Jacobians, and hence the results on rational points on curves, are conditional on GRH. Additionally, since Magma's implementation of elliptic curve arithmetic over $p$-adic fields is not fully numerically stable, we cannot entirely rule out the possibility of an error in precision tracking that compromises the correctness of the computation; however, such errors, even if theoretically possible, are highly unlikely to occur in practice, as this would require unfortunate numerical coincidences at a high degree of precision. At such time as numerically stable $p$-adic elliptic curve arithmetic is implemented in Magma, the computations could be re-run to rule out this possibility.

The runtime and memory requirements seem hard to predict for any given curve, so a time limit of several hours and a memory limit of 8 GB of RAM was set for each curve. Processes that exceeded these limits were terminated. For curves where the computation completed successfully, runtimes appeared to follow a long-tail distribution (Figure \ref{fig:runtimes}); the median runtime was 529 seconds, and the mean was 1145 seconds. For curves where a Mordell--Weil group could not be provably computed (but without timing out) or was found to have too high rank, the distribution of runtimes was similar: median 581 seconds and mean 1250 seconds.

\begin{figure}[ht]
  \centering
  \caption{Histogram of runtimes (in minutes) for the curves where the method succeeded in computing the set of rational points.}
  \includegraphics[scale=0.7]{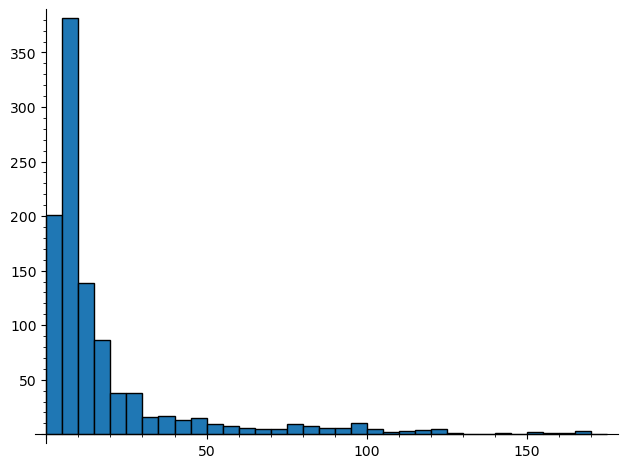}
  \label{fig:runtimes}
\end{figure}

Interestingly, while the success rate decreased for curves with larger discriminant, the average runtimes in the cases where the method succeeded did not appear to significantly increase with the discriminant. Rather, the majority of this decrease was due to an increase in failures of the Hasse principle (see Figure \ref{fig:disc}).

\begin{figure}[ht]
  \centering
  \caption{Portion of curves for which the method succeeded (blue) or encountered an apparent failure of the Hasse principle (red), plotted against the discriminant of the curve (grouped into 10 bins of width $10^5$).}
  \includegraphics[scale=0.7]{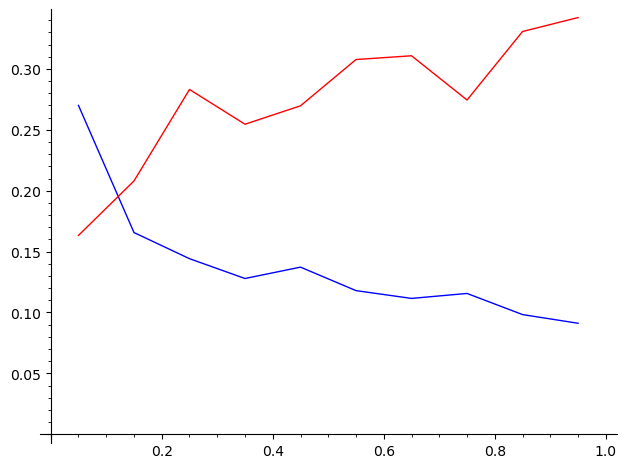}
  \label{fig:disc}
\end{figure}

To reduce the computational resources required, the code was designed to terminate for a given curve as soon as certain obstructions to the success of the computation were detected. Hence, for example, Mordell--Weil groups were not computed when there is an apparent failure of the Hasse principle, so the runtimes for such curves are typically much shorter: a mean of 35 seconds, a median of 17 seconds, and only three such curves having a runtime over 10 minutes.

We also make some observations about the number and height of points on the 4748 genus 5 curves $Z_\delta$ associated to the 1045 genus 2 curves where the method succeeded. The largest cardinality of $Z_\delta(\QQ)$ observed was $6$; the full distribution is shown in Table \ref{tab:point-count}.

\begin{table}[ht]
  \centering
  \caption{Distribution of cardinalities of $Z_\delta(\QQ)$.}
  \begin{tabular}{c|r|r}
    $\#Z_\delta(\QQ)$ & \textbf{Count} & \textbf{Percent} \\
    \hline
    0 & 1136 & 23.9\% \\
    1 & 1602 & 33.7\% \\
    2 & 1531 & 32.2\% \\
    3 & 326 & 6.9\% \\
    4 & 128 & 2.7\% \\
    5 & 18 & 0.4\% \\
    6 & 7 & 0.1\%
  \end{tabular}
  \label{tab:point-count}
\end{table}

We can also analyze the maximum $H_{\mathrm{max}}$ of the naive heights $H(P)$ of points $P \in Z_\delta(\QQ)$ with $Z_\delta$ associated to a genus 2 curve $C$ as above. Among the same set of 1045 genus 2 curves, the median value of the largest coordinate was $16$; the arithmetic and geometric means were approximately $739.8$ and $20.2$, respectively, suggesting a long-tail distribution. The statistic $H_{\mathrm{max}}$ appears to increase gradually with the absolute discriminant $\Delta$ of $C$: a Pearson correlation test on a log-log plot yields a correlation coefficient of $r \approx 0.094$ ($p \approx 0.0023$); see Figure \ref{fig:heights}.

\begin{figure}[ht]
  \centering
  \caption{Log-log plot (base 10) of the absolute discriminant $\Delta$ ($x$-axis) versus the maximum naive height $H_{\mathrm{max}}$ of points in $Z_\delta(\QQ)$ ($y$-axis).}
  \includegraphics[scale=0.7]{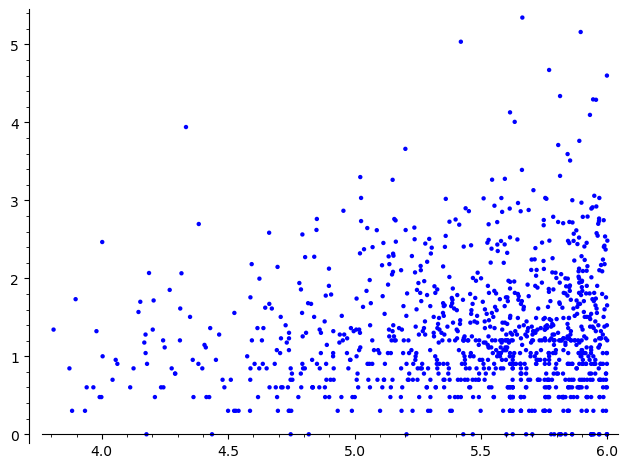}
  \label{fig:heights}
\end{figure}

Let us further note what sort of progress would be necessary to handle the remaining cases:
\begin{enumerate}
  \item In cases where a curve $Z_\delta$ is found to be locally solvable but no rational points can be found, a method of verifying failure of the Hasse principle (such as an implementation of the Mordell--Weil sieve for such curves) would be necessary to proceed.
  \item If one of the elliptic curves has rank greater than or equal to the degree of its base field, then Chabauty's method cannot be applied. In some such cases, Kim's non-abelian generalization of Chabauty's method \cite{Kim09} might be a promising approach.
  \item If Magma is unable to provably compute the Mordell--Weil group of an elliptic curve over a number field within the allotted time, then either an unknown amount more computation time or further advances in descent algorithms for elliptic curves over number fields would be required.
  \item In a small number of cases, either a local solvability test or elliptic Chabauty exceeded the time or memory limits for unclear reasons.
  \item In a handful of cases, Magma threw an exception that suggests a bug in the internal codebase of Magma.
\end{enumerate}

A few more computational remarks:
\begin{enumerate}
    \setcounter{enumi}{5}
  \item If we do not assume GRH, the bottleneck is provably computing the class group of a degree $15$ number field in order to bound the $2$-Selmer rank of the elliptic curves, and this rapidly becomes computationally infeasible as the discriminant grows. (We do carry out the unconditional computation in the first example of the next section.)
  \item When we assume GRH, most of the time is spent either on computing the Mordell--Weil groups of the elliptic curves or on the elliptic Chabauty method.
  \item We use a singular planar model of the curves to quickly test local solvability. Using Proposition \ref{prop:good-reduction}, we only need to check local solvability at the primes of bad reduction of $C$, primes $p \leq 97$ (for which the Hasse--Weil lower bound $\#Z_\delta(\FF_p) \geq p + 1 - 10 \sqrt{p}$ (cf.\ \cite{Milne16}) is non-positive), and the real place. For determining the existence of real points, we use the algorithm of \cite[§4]{SW99}.
\end{enumerate}

\section{Examples}
\label{section:examples}
Let us illustrate the results of the previous sections by examining several examples of successes and failures in detail. The data for the examples in this section was generated using the batch script \texttt{paper-examples.sh} in \cite{Hast22code}; the raw data is available at \cite{Hast22data} in the “examples” folder.

\begin{theorem}
  \label{thm:example1}
  Let $C$ be the genus $2$ curve with LMFDB label \href{https://www.lmfdb.org/Genus2Curve/Q/6443/a/6443/1}{\texttt{6443.a.6443.1}}, which has minimal weighted projective equation
  \[
  C\colon y^2 + z^3 y = x^5 z - x^4 z^2 - 2x^3 z^3 + x^2 z^4 + x z^5.
  \]
  The set of rational points $C(\QQ)$ is
  \begin{gather*}
  \{(1 : 0 : 0), (0 : 0 : 1), (-1 : 0 : 1), (0 : -1 : 1), (1 : 0 : 1), (-1 : -1 : 1), \\
  (1 : -1 : 1), (2 : 2 : 1), (2 : -3 : 1), (-3 : 6 : 4), (-3 : -70 : 4)\}.
  \end{gather*}
\end{theorem}
\begin{proof}
  The change of coordinates $(x : y : z) \mapsto (z : 2y + z^3 : x)$ yields the model
  \[
  y^2 = x^6 + 4x^5 z + 4x^4 z^2 - 8x^3 z^3 - 4x^2 z^4 + 4x z^5,
  \]
  which has a rational Weierstrass point at $(0 : 0 : 1)$. Let $J$ be the Jacobian of $C$. Computing the Mordell--Weil group $J(\QQ)$ in Magma, we find it is free of rank $2$, and applying the Cassels map to representatives of each element of $J(\QQ)/2J(\QQ)$, we obtain four twist parameters $\delta_1, \delta_2, \delta_3, \delta_4$, each corresponding to a genus $5$ curve $Z_\delta$ as in Theorem \ref{thm:embedding}.

  We compute using Magma that $Z_{\delta_4}$ is not locally solvable at $2$, so $Z_{\delta_4}(\QQ) = \emptyset$. For each $i = 1, 2, 3$, we can find a rational point on $Z_{\delta_i}$, so we obtain a map to an elliptic curve $Z_{\delta_i} \to E_i$ over $K = \QQ(\omega)$ (where $\omega$ is a root of $g$), as in Theorem \ref{thm:genus-one}.

  We then compute the Mordell--Weil group of each $E_i$ and apply the elliptic Chabauty method to provably compute the set of $K$-points of each $E_i$ whose image under the given map to $\proj^1$ is rational. To make the computation more efficient, we first compute all four Mordell--Weil groups under the assumption of GRH (which is only used to make class group computations faster), and take note of the number field $F$ whose class group we need to compute, along with the conditionally proven value of its class number $h_F$. By Remark \ref{remark:elliptic-curves}, the number field $F$ and the class number $h_F$ do not depend on $\delta$. Then we compute $h_F$ unconditionally. The results are summarized in Table \ref{tab:example1}.

  \begin{table}[ht]
    \centering
    \caption{Results for genus 2 curve \texttt{6443.a.6443.1}.}
    \begin{tabular}{c|>{\small}c|c|c|c}
       & \normalsize$\delta$ & ELS & $D_{\delta, \omega}(K)$ & $\#Z_\delta(\QQ)$ \\
      \hline
      $\delta_1$ & $1$ & yes & $\ZZ^4$ & $2$ \\
      $\delta_2$ & $X^2 + X - 1$ & yes & $\ZZ^3$ & $3$ \\
      $\delta_3$ & $X^5 + 4X^4 + 4X^3 - 8X^2 - 5X + 4$ & yes & $\ZZ^3$ & $2$ \\
      $\delta_4$ & $-X^5 - 4X^4 - 5X^3 + 7X^2 + 5X - 4$ & no (2) & --- & $0$
    \end{tabular}
    \label{tab:example1}
  \end{table}

  The \enquote{ELS} column indicates whether $Z_\delta$ is everywhere locally solvable, and if not, gives a prime $p$ such that $Z_\delta(\QQ_p) = \emptyset$. The number field whose class group is computed has defining polynomial $x^{15} - 3x^{14} + 15x^{13} - 60x^{12} + 267x^{11} - 1337x^{10} + 2375x^9 - 1676x^8 + 2625x^7 - 4167x^6 - 2687x^5 + 10176x^4 - 4556x^3 - 2616x^2 + 1238x + 406$ over $\QQ$; this field was verified in 24177 seconds to have class number $2$. The other parts of the computation took 1195 seconds in total.

  Next, we apply the map $\bar{\pi}_\delta\colon Z_\delta \to \proj^1$ to each point $P \in Z_\delta(\QQ)$:
  \begin{align*}
    \bar{\pi}_1( (0 : 0 : 0 : 0 : 1) ) &= 0, & \bar{\pi}_{\delta_2}( (22 : 13 : 8 : 2 : 2) ) &= -60/59, \\
    \bar{\pi}_1( (-1 : 0 : -1 : 0 : 2) ) &= 1/2, & \bar{\pi}_{\delta_3}( (1 : 0 : 0 : 0 : 0) ) &= \infty, \\
    \bar{\pi}_{\delta_2}( (2 : 1 : 1 : 0 : 2) ) &= 1, & \bar{\pi}_{\delta_3}( (3 : 2 : 2 : 0 : 4) ) &= -1. \\
    \bar{\pi}_{\delta_2}( (8 : 5 : 4 : 2 : 2) ) &= -4/3,
  \end{align*}
  (Note: we view $Z_\delta$ as embedded in $\proj^4$ with coordinates $v_1, \dots, v_5$. Since $\gamma_6 = f_6 \neq 0$, we can always reconstruct $v_6$ from this information using Theorem \ref{thm:embedding}.) Inverting the change of coordinates on $C$, we see that the set of possible $x$-coordinates of rational points of $C$ is
  \[
  \{\infty, 2, 1, -3/4, -59/60, 0, -1\}.
  \]
  The Weierstrass point lies above $\infty$, and there are two rational points above each of $2, 1, -3/4, 0, -1$, accounting for all $11$ known points in $C(\QQ)$. The two points of $C$ above $-59/60$ are not rational.
\end{proof}

\begin{theorem}
  \label{thm:example2}
  Let $C$ be the genus $2$ curve with LMFDB label \href{https://www.lmfdb.org/Genus2Curve/Q/141991/b/141991/1}{\texttt{141991.b.141991.1}}, which has minimal weighted projective equation
  \[
  C\colon y^2 + (x^2 z + x z^2 + z^3) y = x^5 z - 2x^4 z^2 - 2x^3 z^3 + x^2 z^4.
  \]
  Assuming GRH, the set of rational points $C(\QQ)$ is
  \begin{gather*}
    \{(1 : 0 : 0), (0 : 0 : 1), (-1 : 0 : 1), (0 : -1 : 1), (-1 : -1 : 1), (1 : -1 : 1), \\
    (1 : -2 : 1), (2 : -3 : 1), (2 : -4 : 1), (-1 : 6 : 4), (1 : 6 : 9), (3 : -22 : 4), \\
    (-1 : -58 : 4), (3 : -126 : 4), (1 : -825 : 9)\}.
  \end{gather*}
\end{theorem}
\begin{proof}
  The proof strategy is the same as in the previous example. The change of coordinates $(x : y : z) \mapsto (z : 2y + x^2 z + x z^2 + z^3 : x)$ yields the model
  \[
  y^2 = x^6 + 2x^5 z + 7x^4 z^2 - 6x^3 z^3 - 7x^2 z^4 + 4xz^5,
  \]
  which has a rational Weierstrass point at $(0 : 0 : 1)$. In this case, the Jacobian of $C$ has Mordell--Weil group $\ZZ^3$, so there are $8$ twists to consider. Of these, three have no $\QQ_2$-points and hence no $\QQ$-points, and the rest all have a rational point of low height and are amenable to elliptic Chabauty (with the upper bounds on Mordell--Weil ranks conditional on GRH). The results are summarized in Table \ref{tab:example2}.

  \begin{table}[ht]
    \centering
    \caption{Results for genus 2 curve \texttt{141991.b.141991.1}.}
    \begin{tabular}{c|>{\small}c|c|c|c}
       & \normalsize$\delta$ & ELS & $D_{\delta, \omega}(K)$ & $\#Z_\delta(\QQ)$ \\
      \hline
      $\delta_1$ & $1$ & yes & $\ZZ^3$ & $1$ \\
      $\delta_2$ & $X^2 - 1$ & yes & $\ZZ$\hphantom{${}^1$} & $1$ \\
      $\delta_3$ & $X^5 + 2X^4 + 7X^3 - 5X^2 - 8X + 4$ & yes & $\ZZ^3$ & $2$ \\
      $\delta_4$ & $-X^5 - X^4 - 8X^3 + 5X^2 + 8X - 4$ & yes & $\ZZ^3$ & $2$ \\
      $\delta_5$ & $X^5 + 2X^4 + 7X^3 - 6X^2 - 8X + 4$ & yes & $\ZZ^3$ & $3$ \\
      $\delta_6$ & $-5X^5 - 7X^4 - 27X^3 + 23X^2 + 28X - 16$ & no (2) & --- & $0$ \\
      $\delta_7$ & $4X^5 + 8X^4 + 27X^3 - 23X^2 - 28X + 16$ & no (2) & --- & $0$ \\
      $\delta_8$ & $-X^5 - 2X^4 - 8X^3 + 6X^2 + 8X - 4$ & no (2) & --- & $0$
    \end{tabular}
    \label{tab:example2}
  \end{table}

  The total computation time required was 894 seconds. The number field $F$ whose class group computation depends on GRH has defining polynomial $x^{15} + 6x^{14} + 21x^{13} + 88x^{12} + 212x^{11} + 332x^{10} + 1198x^9 + 3248x^8 + 1626x^7 - 8560x^6 - 3892x^5 - 68524x^4 - 315439x^3 - 494742x^2 - 69455x + 384152$ over $\QQ$, and the class number is $2$ assuming the Bach bound. Verifying this class number would remove the dependence on GRH.

  We apply the map $\bar{\pi}_\delta$ to each point $P \in Z_\delta(\QQ)$:
  \begin{align*}
    \bar{\pi}_1( (0 : 0 : 0 : 0 : 1) ) &= 0, & \bar{\pi}_{\delta_4}( (207 : 82 : 124 : 46 : 106) ) &= 3361/3215, \\
    \bar{\pi}_{\delta_2}( (0 : -1 : 0 : -1 : 1) ) &= 1/2, & \bar{\pi}_{\delta_5}( (1 : 0 : 0 : 0 : 0) ) &= \infty, \\
    \bar{\pi}_{\delta_3}( (1 : 0 : 0 : 0 : 0) ) &= 1, & \bar{\pi}_{\delta_5}( (1 : 1 : 0 : 1 : 1) ) &= 4/3, \\
    \bar{\pi}_{\delta_3}( (-1 : 2 : 2 : 4 : 2) ) &= 9, & \bar{\pi}_{\delta_5}( (2 : 1 : 1 : 0 : 1) ) &= -4. \\
    \bar{\pi}_{\delta_4}( (3 : 4 : 4 : 4 : 4) ) &= -1,
  \end{align*}
  Inverting the change of coordinates, the possible $x$-coordinates for rational points of $C$ are
  \[
  \{\infty, 2, 1, 1/9, -1, 3215/3361, 0, 3/4, -1/4\}.
  \]
  There is the rational Weierstrass point above $\infty$, no rational points above $3215/3361$, and two rational points above each of the others, yielding exactly the $15$ known rational points.
\end{proof}

Now we present a few examples illustrating obstacles the method can encounter.

\begin{example}[Probable failure of the Hasse principle]
  Let $C$ be the genus 2 curve with LMFDB label \href{https://www.lmfdb.org/Genus2Curve/Q/10681/a/117491/1}{\texttt{10681.a.117491.1}}, which has a sextic Weierstrass model
  \[
  C\colon y^2 = 121x^6 - 308x^5 + 276x^4 - 92x^3 + 4x.
  \]
  We compute $J(\QQ) \cong \ZZ^2$. One of the twist parameters we obtain by applying the Cassels map to $J(\QQ)/2J(\QQ)$ is $\delta = -X + 1$. The corresponding genus 5 curve $Z_\delta$ is locally solvable, but the \texttt{PointSearch} function in Magma finds no points on $Z_\delta$ with a bound of $10^6$. (These computations took 15 seconds in total.) Thus, we are unable to provably compute $C(\QQ)$ unless we can prove that $Z_\delta(\QQ)$ is in fact empty.
\end{example}

\begin{example}[Too high rank for elliptic Chabauty]
  Let $C$ be the genus 2 curve with LMFDB label \href{https://www.lmfdb.org/Genus2Curve/Q/7403/a/7403/1}{\texttt{7403.a.7403.1}}, which has a sextic Weierstrass model
  \[
  C\colon y^2 = x^6 + 4x^5 - 4x^4 - 8x^3 + 4x^2 + 4x.
  \]
  We compute $J(\QQ) \cong \ZZ^2$. One of the twist parameters we obtain by applying the Cassels map to $J(\QQ)/2J(\QQ)$ is $\delta = x^5 + 4x^4 - 4x^3 - 7x^2 + 3x + 4$. The corresponding genus 5 curve $Z_\delta$ has three rational points of low height, one of which is $(1 : 0 : 0 : 0 : 0)$, and using this as a base point, we obtain a map $Z_\delta \to E$ defined over the quintic field $K = \QQ(\alpha)$ with $\alpha^5 + 4\alpha^4 - 4\alpha^3 - 8\alpha^2 + 4\alpha + 4 = 0$, where $E$ is the elliptic curve given by
  \begin{align*}
    y^2 &= x^3 + (2\alpha + 4)x^2 + (11\alpha^4 + 57\alpha^3 + 18\alpha^2 - 68\alpha - 34)x \\
    &\qquad+ (36\alpha^4 + 179\alpha^3 + 63\alpha^2 - 211\alpha - 115).
  \end{align*}
  Magma computes that $E(K)$ is free of rank $5$. Thus, we are unable to prove that the three known rational points of $Z_\delta$ are all of the rational points. These computations took 449 seconds in total.
\end{example}

\begin{example}[Unable to compute Mordell--Weil group]
  Let $C$ be the genus 2 curve with LMFDB label \href{https://www.lmfdb.org/Genus2Curve/Q/7211/a/7211/1}{\texttt{7211.a.7211.1}}, which has a sextic Weierstrass model
  \[
  C\colon y^2 = x^6 - 4x^4 + 10x^3 - 8x^2 + 1.
  \]
  We compute $J(\QQ) \cong \ZZ^2$. One of the twist parameters we obtain by applying the Cassels map to $J(\QQ)/2J(\QQ)$ is $\delta = -4x^5 - 4x^4 + 11x^3 - 26x^2 + 3x + 4$. The corresponding genus 5 curve $Z_\delta$ has rational point $(3 : -1 : -1 : -1 : 3)$, and using this as a base point, we obtain a map $Z_\delta \to E$ defined over the quintic field $K = \QQ(\alpha)$ with $\alpha^5 + \alpha^4 - 3\alpha^3 + 7\alpha^2 - \alpha - 1 = 0$, where $E$ is the elliptic curve
  \begin{align*}
    y^2 &= x^3 + (-9\alpha^4 - 13\alpha^3 + 21\alpha^2 - 54\alpha - 18)x^2 \\
    &\qquad + (73\alpha^4 + 110\alpha^3 - 163\alpha^2 + 428\alpha + 144)x \\
    &\qquad + (82336\alpha^4 + 124063\alpha^3 - 184134\alpha^2 + 483038\alpha + 162465).
  \end{align*}
  Magma can compute that the rank of $E(K)$ is at most 1; however, Magma was unable to either find any non-identity $K$-points on $E$ or prove that no such points exist. Thus, we are unable to prove that the list of known rational points of $Z_\delta$ is complete. These computations took 389 seconds in total.
\end{example}

\section*{Acknowledgements}
The author thanks Jennifer Balakrishnan, Raymond van Bommel, Noam Elkies, Brendan Hassett, Steffen M\"uller, Bjorn Poonen, Michael Stoll, Andrew Sutherland, John Voight, and several anonymous reviewers for helpful comments and conversations related to this paper.

\section*{Funding}
This work was supported by the Simons Collaboration on Arithmetic Geometry, Number Theory, and Computation (Simons Foundation grant \#550023).

\section*{Conflicts of interest statement}

The author asserts that there are no conflicts of interest.

\section*{Data availability statement}

The datasets generated as part of this work are available at \cite{Hast22data}. The code used to generate the datasets is available at \cite{Hast22code}.

\bibliography{two-cover-descent}

\end{document}